\documentclass[a4paper,12pt, british, reqno]{amsart}

\usepackage{babel}
\usepackage{mathabx}
\usepackage{amssymb}
\usepackage{amsfonts}
\usepackage{enumitem}
\usepackage{hyperref}
\usepackage[utf8]{inputenc}
\usepackage{newunicodechar}
\usepackage{mathtools}
\usepackage{varioref}
\usepackage[centering]{geometry}
\usepackage[arrow,curve,matrix]{xy}
\usepackage{csquotes}
\usepackage{dynkin-diagrams}

\usepackage{colortbl}
\usepackage{graphicx}
\usepackage{tikz}
\usepackage{tikz-cd}

\usepackage{mathrsfs}

\usepackage{multirow}
\usepackage{multicol}
\usepackage{longtable} 
\usepackage{caption}
\usepackage{hhline}
\newcolumntype{M}[1]{>{\centering\arraybackslash}m{#1}} 

\definecolor{linkred}{rgb}{0.7,0.2,0.2}
\definecolor{linkblue}{rgb}{0,0.2,0.6}

\numberwithin{figure}{section}

\setdescription{labelindent=\parindent, leftmargin=2\parindent}
\setitemize[1]{labelindent=\parindent, leftmargin=2\parindent}
\setenumerate[1]{labelindent=0cm, leftmargin=*, widest=iiii}

\DeclareFontFamily{OMS}{rsfs}{\skewchar\font'60}
\DeclareFontShape{OMS}{rsfs}{m}{n}{<-5>rsfs5 <5-7>rsfs7 <7->rsfs10 }{}
\DeclareSymbolFont{rsfs}{OMS}{rsfs}{m}{n}
\DeclareSymbolFontAlphabet{\scr}{rsfs}
\DeclareSymbolFontAlphabet{\scr}{rsfs}

\DeclareMathOperator{\codim}{codim}

\DeclareMathOperator{\rank}{rank}

\DeclareMathOperator{\reg}{reg}

\DeclareMathOperator{\Spec}{Spec}

\DeclareMathOperator{\gr}{gr}

\DeclareMathOperator{\Cl}{Cl}

\DeclareMathOperator{\ssc}{sc}

\newcommand{\sD}{\scr{D}}

\newcommand{\sF}{\scr{F}}

\newcommand{\sO}{\scr{O}}

\newcommand{\cO}{\mathcal O}
\newcommand{\cR}{\mathcal R}

\newcommand{\0}{\mathcal O}
\newcommand{\bA}{\mathbb{A}}

\newcommand{\bG}{\mathbb{G}}

\newcommand{\bK}{\mathbb{K}}

\newcommand{\bP}{\mathbb{P}}
\newcommand{\bQ}{\mathbb{Q}}

\newcommand{\bT}{\mathbb{T}}

\newcommand{\bZ}{\mathbb{Z}}

\newcommand{\fg}{\mathfrak{g}}

\newcommand{\fp}{\mathfrak{p}}

\newcommand{\fu}{\mathfrak{u}}

\newcommand{\fX}{\mathfrak{X}}

\newcommand{\aS}{{\sf S}}

\theoremstyle{plain}
\newtheorem{thm}{Theorem}[section]

\newtheorem{conjecture}[thm]{Conjecture}
\newtheorem{cor}[thm]{Corollary}
\newtheorem{defn}[thm]{Definition}

\newtheorem{lem}[thm]{Lemma}

\newtheorem{ques}[thm]{Question}
\newtheorem{prop}[thm]{Proposition}

\theoremstyle{remark}

\newtheorem{c-n-d}[thm]{Claim and Definition}

\newtheorem{example}[thm]{Example}

\newtheorem{rem}[thm]{Remark}

\newtheorem*{rem-nonumber}{Remark}

\numberwithin{equation}{thm}

\setlist[enumerate]{label=(\thethm.\arabic*), before={\setcounter{enumi}{\value{equation}}}, after={\setcounter{equation}{\value{enumi}}}}

\newcommand{\factor}[2]{\left. \raise 2pt\hbox{$#1$} \right/\hskip -2pt\raise -2pt\hbox{$#2$}}

\author{Baohua Fu}
\address{Baohua Fu, State Key Laboratory of Mathematical Sciences, Morningside Center of Mathematics, Academy of Mathematics and Systems Science, Chinese Academy of Sciences, Beijing 100190, China;   and School of Mathematical Sciences, University of Chinese Academy of Sciences, Beijing, China}
\email{\href{bhfu@math.ac.cn}{bhfu@math.ac.cn}}
\urladdr{\href{http://www.math.ac.cn/people/fbh/}{http://www.math.ac.cn/people/fbh/}}

\author{Jie Liu} %
\address{Jie Liu, Institute of Mathematics, Academy of Mathematics and Systems Science, Chinese Academy of Sciences, Beijing, 100190, China}
\email{\href{jliu@amss.ac.cn}{jliu@amss.ac.cn}}
\urladdr{\href{http://www.jliumath.com}{http://www.jliumath.com}}


\makeatletter
\@namedef{subjclassname@2020}{2020 Mathematics Subject Classification}
\makeatother

\title[]{The affine closure of cotangent bundles of horospherical spaces}
\date{\today}

\makeatletter

\hypersetup{
	pdfauthor={\authors},
	pdftitle={\@title},
	pdfsubject={\@subjclass},
	pdfkeywords={\@keywords},
	pdfstartview={Fit},
	pdfpagemode={UseOutlines},
	pdfpagelayout={OneColumn},
	colorlinks,
	linkcolor=linkblue,
	citecolor=linkred,
	urlcolor=linkred
}
\makeatother

\DeclareMathOperator{\GL}{GL}

\DeclareMathOperator{\SL}{SL}

\begin{document}
	
	\begin{abstract}
        For a smooth quasi-affine variety $X$, the affine closure $\overline{T^*X} \coloneqq \Spec(\bK[T^*X])$ contains $T^*X$ as an open subset, and its smooth locus carries a symplectic structure. A natural question is whether $\overline{T^*X}$ itself is a symplectic variety. A notable example is the conjecture of Ginzburg and Kazhdan, which predicts that $\overline{T^*(G/U)}$ is a symplectic variety for a maximal unipotent subgroup $U$ in a reductive linear algebraic group $G$. This conjecture was recently proved by Gannon using representation-theoretic methods. 

        In this paper, we provide a new geometric approach to this conjecture. Our method allows us to prove a more general result: $\overline{T^*(G/H)}$ is a symplectic variety for any horospherical subgroup $H$ in $G$ such that $G/H$ is quasi-affine. In particular, this implies that the affine closure $\overline{T^*(G/[P,P])}$ is a symplectic variety for any parabolic subgroup $P$ in $G$.
	\end{abstract}

	\maketitle
	\tableofcontents
	
	\section{Introduction}

    Throughout this paper, we work over an algebraically closed field $\bK$ of characteristic zero. For an irreducible normal variety $Y$, the $\bK$-algebra 
   $\bK[Y]$ of regular functions on $Y$ is an integrally closed domain, and the \emph{affinization} (or \emph{affine closure}) of $Y$ is defined as the affine normal scheme $\overline{Y} \coloneqq \operatorname{Spec} \bK[Y]$. 
    
	\subsection{Affine closure of cotangent bundles}
	
Let $X$ be a smooth variety. Recall that the algebra $\bK[T^*X]$ is canonically isomorphic to the algebra of symmetric tensors
\[
S(X) \coloneqq \bigoplus_{p \geq 0} H^0(X, \aS^p T_X) = \bigoplus_{p \geq 0} H^0(\mathbb{P}T_X, \mathscr{O}_{\mathbb{P}T_X}(p)). \footnote{Throughout this paper, the projectivization $\bP E$ of a vector bundle $E$ is defined in the sense of Grothendieck, i.e. $\mathbb{P}E \coloneqq \operatorname{Proj}_X \aS^\bullet E$.}
\]
Then there exists a natural dominant affinization morphism
\[
\varphi_X: T^*X \to \overline{T^*X}=\operatorname{Spec} (S(X)),
\]
which is generally non-surjective. Note that in general, $S(X)$ is not finitely generated, even for projective manifolds (see Section~\ref{e.nonfinitegenerated}). However, despite its straightforward definition, the affine scheme $\overline{T^*X}$ and the morphism $\varphi_X$ have received limited attention in the literature (cf.~\cite{BeauvilleLiu2024,FuLiu2025}). In this context, the following fundamental questions naturally arise:
	
	\begin{ques} \label{q.1}
		\begin{enumerate}
			\item When is $S(X)$ finitely generated (over $\bK$)?
			
			\item What kind of singularities can $\overline{T^*X}$ have?
		\end{enumerate}
	\end{ques}

The canonical symplectic structure on $T^*X$ induces a Poisson structure on $S(X)$ known as the \emph{Schouten--Nijenhuis bracket}, making $\overline{T^*X}$ a normal integral affine Poisson scheme, so it is natural to consider the notion of symplectic singularities as introduced by Beauville~\cite{Beauville2000a}. Recall that a normal variety $W$ is \emph{symplectic} if there exists a symplectic form $\omega$ on the smooth locus $W_{\text{reg}}$ of $W$ such that for any resolution $Z \to W$, the $2$-form $\omega$ extends to a regular $2$-form on $Z$. 

If $X$ is affine, then $T^*X$ is also affine, and thus $\varphi_X$ becomes an identity, implying that $\overline{T^*X} = T^*X$ is again a smooth symplectic variety. When $X$ is projective, in~\cite{FuLiu2025} it has been shown  that $\varphi_X$ is birational if and only if $T_X$ is big. In this case, the affine closure $\overline{T^*X}$ is a symplectic variety if and only if $\mathbb{P} T_X$ is of Fano type. A typical example is when $X$ is a projective rational homogeneous space; in this case, $\varphi_X$ is the Stein factorization of the moment map and $\overline{T^*X}$ is a finite cover of a nilpotent orbit closure in $\mathfrak{g}$.

In this paper, we investigate Question~\ref{q.1} for quasi-affine horospherical spaces, motivated by the study of the \emph{basic affine space} $G/U$, where $G$ is a connected reductive group and $U$ is a maximal unipotent subgroup. The affine closure $\overline{T^*(G/U)}$, called \emph{universal implosion}, plays an important role in the Dancer--Kirwan--Swann theory of hyperk\"ahler implosions \cite{DancerKirwanSwann2013} as well as in the context of $3d \ \mathcal{N} = 4$ quiver gauge theory \cite{BourgetDancerGrimmingerHananyZhong2022}. The finite generation of $S(G/U)$ was proved for $G=\SL_n$ in \cite{DancerKirwanSwann2013} and  for general reductive groups in \cite{GinzburgRiche2015}. Moreover, V.~Ginzburg and D.~Kazhdan conjectured in~\cite[Conjecture~1.3.6]{GinzburgKazhdan2022} that $\overline{T^*(G/U)}$ is a symplectic variety.  This conjecture has been verified by B.~Jia in~\cite{Jia2025a} for $G = \mathrm{SL}_n$ and by T.~Gannon in~\cite{Gannon2024} for general $G$. The key idea of their proofs is to use geometric representation theory to show that the codimension of the singular locus is at least four, which in fact establishes the stronger result that $\overline{T^*(G/U)}$ has terminal symplectic singularities.
	
	\subsection{Quasi-affine horospherical spaces}
	
	Note that $U=[B,B]$ for the Borel subgroup $B$ of $G$ containing $U$, thus it is natural to ask whether the Ginzburg--Kazhdan conjecture still holds for the quasi-affine homogeneous space $G/[P, P]$, where $P$ is a parabolic subgroup of $G$. We observe that $G/[P,P]$ is an example of so-called \emph{horospherical spaces}. Here, recall that a homogeneous space $G/H$ is called \emph{horospherical} if $G$ is reductive and the normalizer $N_G(H)$ is parabolic such that $N_G(H)/H$ is a torus, or equivalently $H$ contains a maximal unipotent subgroup $U$ of $G$ \cite[\S\,7]{Timashev2011}.  
	
	Our first result confirms a generalization of Ginzburg--Kazhdan's conjecture for quasi-affine horospherical homogeneous spaces.

\begin{thm}
    \label{t.MainThmHorospherical}
    Let $G$ be a connected reductive group, and let $H \subset G$ be a horospherical subgroup such that $G/H$ is quasi-affine. Then the $\bK$-algebra $S(G/H)$ is finitely generated and $\overline{T^*(G/H)}$ is a symplectic variety.
\end{thm}

As a corollary, we obtain the following result:

\begin{cor}
    \label{c.U_P/[P,P]-Cotangent}
    Let $G$ be a connected reductive group, and let $P$ be a parabolic subgroup of $G$.
    \begin{enumerate}[label=(\arabic*)]
        \item The $\bK$-algebra $S(G/[P, P])$ is finitely generated, and $\overline{T^*(G/[P, P])}$ is a $\mathbb{Q}$-factorial symplectic variety.
        
        \item If $[G,G]$ is simply connected, then $\overline{T^*(G/[P, P])}$ is factorial.
    \end{enumerate}
\end{cor}

Unlike previous works on $G/U$ in the literature, our approach to Theorem \ref{t.MainThmHorospherical} has a geometric nature, following a strategy developed in our previous work \cite[\S\,5]{FuLiu2025}. The key observation is that $S(G/H)$ can be realized as the invariant ring of the Cox ring of $G/H$ under the action of a quasi-torus and the latter one is isomorphic to the Cox ring of a weak Fano manifold (cf.~\cite[\S\,5.2]{FuLiu2025}).

\subsection{Cones of highest weight vectors}

The simplest example of a quasi-affine horospherical space is the punctured affine cone $X$ of an equivariant embedding $Z\subset \bP^N$ of a rational homogeneous space (see \S\,\ref{ss.HVcones}). However, it is a very difficult problem to give an explicit description of the singularities of $\overline{T^*X}$ even for a variety $Z$ as simple as the quadric. For example, the affine closure of $\SL_3/U$ is isomorphic to the quadric cone $\widehat{\bQ}^5\subset \bA^6$ so that $\overline{T^*(\SL_3/U)}\cong \overline{T^*X}$, where $X\coloneqq \widehat{\bQ}^5\setminus\{0\}$ is the $5$-dimensional punctured quadric cone, therefore \cite[Corollary 3.8]{Jia2025a} implies that $\overline{T^*X}$ is isomorphic to the minimal nilpotent orbit closure $\overline{\cO}_{\min}$ in $\mathfrak{so}_8$. 

 Our second result gives a similar description for certain cones of highest weight vectors associated with irreducible Hermitian symmetric spaces.

\begin{thm}
    \label{thm.coneIHSS}
    Let $G$ be a connected simple group, and let $G/P$ be an irreducible Hermitian symmetric space  different from projective spaces. Let $L$ be the Levi subgroup of $P$ with $v$ a highest weight vector in the nilradical $\mathfrak{u}_P \subset \mathfrak{p}$ for $L$. Let $X\coloneqq L\cdot v \subset \mathfrak{u}_P$ be the cone of highest weight vectors. Then $\overline{T^*X}$ is isomorphic to the minimal nilpotent orbit closure $\overline{\mathcal{O}}_{\text{min}} \subset \mathfrak{g}$.
\end{thm}

We refer the reader to Table \ref{Table.IHSS-Cone} for a list of $X$  that appeared in Theorem \ref{thm.coneIHSS}. An immediate consequence of Theorem \ref{thm.coneIHSS} is that for a $n$-dimensional punctured quadric cone $X$, the affine closure $\overline{T^*X}$ is isomorphic to $\overline{\cO}_{\min}\subset \mathfrak{so}_{n+3}$, generalizing Jia's result in dimension $5$ mentioned above (see Example \ref{e.quadric}), and if $\mathfrak{g}$ is not of type $G_2$, $F_4$ or $E_8$, then $\mathcal{O}_{\text{min}}$ contains a large open subset isomorphic to $T^*X$, which is somewhat surprising. We note that employing different methods, the minimal nilpotent orbit closures in $\mathfrak{so}_{2n}$ (\cite{Jia2025}) and $\mathfrak{e}_6$ (\cite{Jia2025,GannonWebster2025}) have also been realized as the affine closures of cotangent bundles (distinct from ours). 

Inspired by the results above, we propose the following question:

\begin{ques} \label{q.2}
	Let $Z \subset \mathbb{P}^N$ be a smooth projective variety and let $X\subset \bA^{N+1}$ be its punctured affine cone. When is $\overline{T^*X}$ a symplectic variety?
\end{ques}

	\subsection{Parabolic basic affine spaces}

    Another natural generalization of basic affine spaces is the \emph{parabolic basic affine spaces} $G/U_P$, where $U_P$ is the unipotent radical of a parabolic subgroup $P$ of a connected reductive group $G$. It is expected (cf. \cite[Section 8]{DancerHananyKirwan2021} and \cite{BourgetDancerGrimmingerHananyZhong2022}) that $S(G/U_P)$ is finitely generated and that $\overline{T^*(G/U_P)}$, referred to as the \emph{partial implosion}, is symplectic. This has been confirmed for $G = \SL_n$ or $\GL_n$ by T.~Gannon and B.~Webster in \cite{GannonWebster2025}. 

 Following the idea in the proof of Theorem \ref{t.MainThmHorospherical}, we reformulate the above expectation from the perspective of birational geometry as follows. For the definitions of Mori dream spaces and varieties of Fano type we refer to \S\,\ref{ss.DefnCoxRing}.

\begin{thm} \label{t.G/UP}
    Let $G$ be a connected semisimple group, and let $U_P$ be the unipotent radical of a parabolic subgroup $P$ of $G$. Denote by $E$ the homogeneous vector bundle $G\times_B \fu_P^{\perp}$ over $G/B$.
    \begin{enumerate}
        \item\label{i1.t.G/UP} The $\bK$-algebra $S(G/U_P)$ is finitely generated if and only if $\bP E^*$ is a Mori dream space.
        
        \item\label{i2.t.G/UP} The variety $\overline{T^*(G/U_P)}$ is symplectic if and only if $\mathbb{P} E^*$ is of Fano type.
    \end{enumerate}
\end{thm}

Based on Theorem \ref{t.G/UP}, we propose the following conjecture, which is supported by the evidence provided in Proposition \ref{p.weakfano}.

\begin{conjecture}
    \label{conj.PEFanoType}
    The projective bundle $\mathbb{P} E^*$ is of Fano type.
\end{conjecture}

    \subsection*{Acknowledgments}
	  We are very grateful to the anonymous referees for the pertinent comments and suggestions that helped us improve this paper.   
    We would like to thank B.~Jia for useful discussions and I.~Losev for drawing our attention to the works of T.~Gannon. We are grateful to A.~Beauville, T.~Gannon and S.~Takagi for providing valuable feedback on the preliminary version of this paper. J.~Liu is supported by the National Key Research and Development Program of China (No. 2021YFA1002300) and the Youth Innovation Promotion Association CAS. Both authors are supported by the CAS Project for Young Scientists in Basic Research (No. YSBR-033) and the NSFC grant (No. 12288201).

	\section{Preliminaries}

    All varieties are assumed to be integral schemes of finite type over $\bK$, and all groups are assumed to be affine algebraic groups over $\bK$. We denote by $\bG_m$ the multiplicative group of $\bK$. Rings are assumed to be commutative with identity.
	
	\subsection{Cox rings and Mori dream spaces}
	\label{ss.DefnCoxRing}
	
	Let $X$ be a normal variety with a finitely generated free divisor class group $\Cl(X)$. Let $D_1, \dots, D_\rho$ be Weil divisors whose classes form a $\mathbb{Z}$-basis of $\Cl(X)$. The \emph{Cox ring} of $X$ is defined as the following $\Cl(X)$-graded ring:
	\[
	\cR(X) \coloneqq \bigoplus_{(m_1, \dots, m_\rho) \in \mathbb{Z}^\rho} H^0(X, \sO_X(m_1 D_1 + \dots + m_\rho D_\rho)).
	\]
	This definition can be extended to normal varieties $X$ with a finitely generated divisor class group $\Cl(X)$ that may have torsion, provided that $H^0(X, \sO_X^\times) = \bK^{\times}$. For further details on this construction, we refer the reader to \cite[\S\,1.4.2]{ArzhantsevDerenthalHausenLaface2015}.

	\begin{defn}
		\label{defn.MDS-Fano}
		Let $X$ be a normal projective variety.
		\begin{enumerate}
			\item We call $X$ a Mori dream space if $X$ is $\bQ$-factorial  with finitely generated $\Cl(X)$ and finitely generated Cox ring $\cR(X)$.
			
			\item We call $X$ of Fano type if there exists an effective $\bQ$-divisor such that $(X,\Delta)$ has Kawamata log terminal singularities (klt singularities) and $-(K_X+\Delta)$ is ample.
		\end{enumerate}
	\end{defn}
	
	We refer to \cite[\S\,2.3]{KollarMori1998} for the definition of singularities of a log pair. According to \cite[Corollary 1.3.2]{BirkarCasciniHaconMcKernan2010}, a $\bQ$-factorial projective variety of Fano type is always a Mori dream space.
	
	\begin{example}
		\label{ex.WeakFano}
		Recall that a projective manifold is called \emph{weak Fano} if $-K_X$ is big and nef. It is known that a weak Fano manifold is of Fano type. Indeed, since $-K_X$ is nef and big, there exists an effective $\bQ$-divisor $D$ such that $-K_X - \frac{1}{m} D$ is ample for all $m\gg 1$ (see \cite[Proposition 2.61]{KollarMori1998}). Moreover, by the smoothness of $X$, the pair $(X, \frac{1}{m} D)$ has klt singularities for all $m\gg 1$.
	\end{example}
	
		\subsection{Examples of non-finite generation}
	\label{e.nonfinitegenerated}
	
	We now give examples of smooth varieties whose algebra of symmetric tensors is not finitely generated.
	
	\begin{example}
		Let $X$ be a smooth variety such that $\bK[X]$ is not finitely generated over $\bK$. Let $\sigma \colon X \rightarrow T^*X$ be the zero section of $\pi \colon T^*X \rightarrow X$. Then the following composition of morphisms of $k$-algebras
		\[
		\bK[X] \stackrel{\pi^*}{\longrightarrow} \bK[T^*X] \stackrel{\sigma^*}{\longrightarrow} \bK[X]
		\]
		coincides with the identity. This implies that $\bK[T^*X]$ is not finitely generated.
	\end{example}
	
	\begin{example}
		Let $Y$ be a smooth projective variety with a nef and big line bundle $L$ on $Y$ such that the ring $R(Y, L)$ of sections is not finitely generated. Let $E$ be the rank two vector bundle $L^{-1} \oplus L$ over $Y$. Denote by $X$ the projective bundle $\pi \colon \mathbb{P}(E) \rightarrow Y$ and by $\sigma \colon Y \rightarrow \mathbb{P}(E)$ the section corresponding to the natural quotient $E \rightarrow L$. Let $Z\coloneqq \sigma(Y)$. Then we have $T_X|_Z\cong T_Z\oplus T_{X/Y}|_Z$, which yields
		\[
		\sigma^* T_X\cong \sigma^*T_Z\oplus \sigma^* T_{X/Y} \cong T_Y \oplus \sigma^* \sO_{\mathbb{P}(E)}(2) \cong T_Y \oplus L^{\otimes 2}.
		\]
		Consider the following composition of morphisms of $\bK$-algebras
		\[
		q \colon \bK[T^*X] \cong S(X) \longrightarrow \bigoplus_{p\geq 0} H^0(Y, \sigma^*\aS^p T_X) \longrightarrow \bigoplus_{p \geq 0} H^0(Y, \aS^p L^{\otimes 2}) = R(Y, L^{\otimes 2}).
		\]
		For any integer $p \geq 0$, the natural morphism
		\[
		\aS^{2p} E = \pi_* \sO_{\mathbb{P}(E)}(2p) = \sigma^* \pi^* \pi_* \sO_{\mathbb{P}(E)}(2p) \rightarrow \sigma^* \sO_{\mathbb{P}(E)}(2p) \cong L^{\otimes 2p}
		\]
		coincides with the natural projection $\aS^{2p} E \rightarrow \aS^{2p} L = L^{\otimes 2p}$. Thus, the natural inclusion
		\[
		\bigoplus_{p \geq 0} H^0(Y, \aS^{2p} E) = \bigoplus_{p \geq 0} H^0(X, T_{X/Y}^{\otimes p}) \longrightarrow \bK[T^*X]
		\]
		shows that $q$ is surjective and therefore $\bK[T^*X]$ is not finitely generated.
	\end{example}
	
	\subsection{Affine closure}
	
	We collect some basic facts on affine closures, which will be used in the latter. Firstly, we recall:
	
	\begin{thm}[\protect{\cite[Theorem 4.2]{Grosshans1997}}]
		\label{Thm.Grosshans-SmallBoundary}
		Let $X$ be a smooth quasi-affine variety. If $\bK[X]$ is finitely generated, then the natural morphism $X \rightarrow \overline{X}$ is an open embedding with a small boundary, i.e., $\codim(\overline{X} \setminus X) \geq 2$.
	\end{thm}
	
	\begin{proof}
		Since $\bK[X]$ is finitely generated, its affine closure $\overline{X}$ is an irreducible normal affine variety such that $X\rightarrow \overline{X}$ is an open embedding with $\bK[X]=\bK[\overline{X}]$; then the statement follows from \cite[Theorem 4.2]{Grosshans1997}.
	\end{proof}

	\begin{prop}
		\label{prop.FiniteCover}
		Let $f \colon X \rightarrow Y$ be a finite morphism between smooth quasi-affine varieties. 
		\begin{enumerate}
			\item\label{i1.Prop.FiniteCover} If $\bK[Y]$ is finitely generated, then the induced morphism $\overline{X} \rightarrow \overline{Y}$ is finite. In particular, $\bK[X]$ is finitely generated.
			
			\item\label{i2.Prop.FiniteCover} If $f$ is Galois, then $\bK[X]$ is finitely generated if and only if $\bK[Y]$ is finitely generated.
		\end{enumerate}
	\end{prop}
	
	\begin{proof}
		For \ref{i1.Prop.FiniteCover}, let $\iota \colon Y \rightarrow \overline{Y}$ be the natural open embedding. Since $f$ is proper and finite, the push-forward $f_*\sO_X$ is coherent and reflexive by \cite[Corollary 1.7]{Hartshorne1980}. As $\codim(\overline{Y} \setminus Y) \geq 2$ by Theorem \ref{Thm.Grosshans-SmallBoundary}, the sheaf $\iota_*f_*\sO_X$ is coherent and reflexive by Lemma \ref{lem.push-forward-coherence} below, which implies that 
		\[
		\bK[X] = H^0(Y, f_*\sO_X) = H^0(\overline{Y}, \iota_*f_*\sO_X)
		\]
		is a finite $\bK[Y]$-module.
		
		For \ref{i2.Prop.FiniteCover}, write $Y=X/\Gamma$ for a finite group $\Gamma$. Then $\bK[Y]=\bK[X]^\Gamma$ and it follows that $\bK[Y]$ is finitely generated if $\bK[X]$ is finitely generated.
	\end{proof}  
	
	The following result is well-known for experts; we include a proof for the reader's convenience.
	
	\begin{lem}
		\label{lem.push-forward-coherence}
		Let $X$ be a normal variety and let $\iota\colon X^\circ\rightarrow X$ be an open embedding such that $\codim(X\setminus X^\circ)\geq 2$. If $\sF$ is a coherent reflexive sheaf on $X^\circ$, then $\iota_*\sF$ is a coherent reflexive sheaf on $X$.
	\end{lem}
	
	\begin{proof}
		Let $\sF'$ be a coherent sheaf on $X$ such that $\sF'|_{X^\circ}\cong \sF$ \cite[II, Exercise 5.15]{Hartshorne1977}. Denote by $\sF''$ the double dual of $\sF'$. Then $\sF''$ is a coherent reflexive sheaf by \cite[Corollary 1.2]{Hartshorne1980} such that $\sF''|_{X^\circ}\cong \sF$ because $\sF$ itself is reflexive. As $\codim(X\setminus X^\circ)\geq 2$, it follows from \cite[Proposition 1.6 (iii)]{Hartshorne1980} that $\sF''\cong \iota_*\sF$. Hence $\iota_*\sF$ is a coherent reflexive sheaf.
	\end{proof}

	\begin{example}
		Let $M$ be a smooth affine variety on which a finite group $\Gamma$ acts freely on an open subset $M_0$ with $\codim(M\setminus M_0)\geq 2$. Let $Z = M / \Gamma$ be the quotient and let $X \subset Z$ be its smooth locus. Then $X$ contains $M_0 / \Gamma$ with a boundary of codimension $\geq 2$. Hence, $\bK[T^*X] \cong \bK[T^*(M_0 / \Gamma)]$, and since $T^*(M_0 / \Gamma)$ is an open subset of $(T^*M)/\Gamma$ with codimension $\geq 2$ boundary, we have $\overline{T^*X} \cong (T^*M)/\Gamma$.
	\end{example}

	\section{Quasi-affine horospherical spaces}
	\label{s.Horospherical}
	
	In this section, we prove Theorem \ref{t.MainThmHorospherical}, following the idea of \cite[\S\,5.2]{FuLiu2025}. The following result is a special case of \cite[Theorem 4.4.1.2]{ArzhantsevDerenthalHausenLaface2015} and  will play a central role in the proof. 
	
	\begin{thm}
		\label{t.Coxring-torusquotient}
		Let $X$ be a smooth variety with finitely generated $\Cl(X)$ and $H^0(X, \sO_X^\times) = \bK^\times$. Assume that there exists a free algebraic torus action $\mathbb{T} \times X \rightarrow X$ that admits a smooth geometric quotient $Y=X/\mathbb{T}$. Then $H^0(Y, \sO_Y^\times) = \bK^\times$, $\Cl(Y)$ is finitely generated, and there is a $\Cl(X)$-graded isomorphism 
		\[
		\cR(Y) \longrightarrow \cR(X).
		\]
	\end{thm}
	
	From now on, let $X$ be a horospherical space $G/H$ so that $T^*X$ can be canonically identified with the homogeneous bundle $G \times_H \mathfrak{h}^{\perp}$, where $\mathfrak{h}^{\perp} \subset \mathfrak{g}^*$ is the annihilator of $\mathfrak{h}$. Let $P\coloneqq N_G(H)$ be the associated parabolic subgroup. Then the quotient $P/H$ is a torus that acts on $G/H$ by right multiplication, which is free and transitive over the fibers $G/H\rightarrow G/P$. Then we have a $G$-equivariant map
	\[
	\phi \colon T^*X \cong G \times_H \mathfrak{h}^{\perp} \longrightarrow G \times_P \mathfrak{h}^{\perp} \eqqcolon E,
	\]
	where $E$ is a homogeneous bundle over the homogeneous space $G/P$. 
	
	Recall that a $\bK$-algebra $R$ has \emph{rational singularities} if it is finitely generated and the associated affine variety $\Spec R$ has rational singularities. 
	
	\begin{lem} \label{l.PE}
		The Cox ring $\mathcal{R}(\mathbb{P} E^*)$ is a unique factorization domain (UFD) and has rational singularities.
	\end{lem}
	
	\begin{proof}
		By \cite[Lemma 5.3]{FuLiu2025}, the projective manifold $\mathbb{P} E^*$ is a weak Fano manifold and hence is of Fano type by Example \ref{ex.WeakFano}, then it follows from \cite[Theorem 1.1]{GongyoOkawaSannaiTakagi2015} that $\cR(\bP E^*)$ is finitely generated and $\Spec(\cR(\bP E^*))$ has klt singularities, therefore $\cR(\bP E^*)$ has rational singularities by \cite[Theorem 5.22]{KollarMori1998}. 
		
		Finally, note that the divisor class group of $\bP E^*$
        is a finitely generated free abelian group, so \cite[Corollary 1.2]{ElizondoKuranoWatanabe2004} implies that $\cR(\bP E^*)$ is a UFD.
	\end{proof}

	 The induced $(P/H)$-action on $T^*X$ yields a free and transitive $(P/H)$-action on the fibers of 
	\begin{equation}
		\label{eq.QuotientMap}
		\overline{\phi}\colon \bP T_X \longrightarrow \bP E^*.
	\end{equation}
	Theorem \ref{t.MainThmHorospherical} can easily be derived from the following result.
	
	\begin{prop}
		\label{prop.GeneralHoro}
		Let $G$ be a connected reductive group, and let $X = G/H$ be a horospherical  space.
		\begin{enumerate}
			\item\label{i1.fg} The $\bK$-algebra $S(X)$ has rational singularities.
			
			\item\label{i2.factorial} The affinization $\overline{T^*X}$ is factorial if $\Cl(X) = 0$ and $\mathbb{Q}$-factorial if $\Cl(X)$ is finite.
		\end{enumerate}
	\end{prop}
	
	\begin{proof}
		Let $\pi\colon \bP T_X\rightarrow X$ be the natural projection. Note that the divisor class group $\Cl(\bP T_X)\cong \pi^*\Cl(X)\oplus \bZ\sO_{\bP T_X}(1)$ is finitely generated, because $\Cl(X)$ is finitely generated (see \cite[Corollary 4.5.14]{ArzhantsevDerenthalHausenLaface2015}).
		
		Firstly, we assume that $H^0(X,\sO_X^{\times})\not=\bK^{\times}$. Then there exists an $G$-equivariant isomorphism $X\cong X'\times A$, where $A$ is a torus with trivial $G$-action and $X'$ is a horospherical $G$-variety with $H^0(X',\sO_X^{\times})=\bK^{\times}$. Then $\Cl(X)\cong \Cl(X')$ and the K\"unneth formula gives an isomorphism $S(X) \cong S(X')\otimes_{\bK} S(A)$, thus yielding
		\[
		\overline{T^*X} \cong \overline{T^*X'} \times \overline{T^*A} = \overline{T^*X'} \times T^*A \cong \overline{T^*X'}\times A \times \bA^r,
		\]
		where $r=\dim A$. Consequently, the $\bK$-algebra $S(X)$ is finitely generated (resp. has rational singularities) if and only if $S(X')$ is finitely generated (resp. has rational singularities). Moreover, the natural isomorphism $\Cl(\overline{T^*X})\cong \Cl(\overline{T^*X'})$ then implies that $\overline{T^*X}$ is factorial (resp.\ $\bQ$-factorial) if and only if $\overline{T^*X'}$ is. 
		
		 Thanks to the above arguments, we may assume that $H^0(X,\sO_X^{\times})=\bK^{\times}$, after replacing $X$ by $X'$ if necessary, so that the Cox ring $\cR(\bP T_X)$ is well-defined. Then applying Theorem \ref{t.Coxring-torusquotient} to the geometric quotient $\overline{\phi}\colon \mathbb{P} T_X \rightarrow \mathbb{P} E^*$ in \eqref{eq.QuotientMap} under the $(P/H)$-action yields a graded isomorphism 
		\[
		\mathcal{R}(\mathbb{P} T_X) \cong \mathcal{R}(\mathbb{P} E^*).
		\]
		In particular, since the cokernel of $\mathbb{Z} \sO_{\mathbb{P} T_X}(1) \rightarrow \Cl(\mathbb{P} T_X)$ is naturally isomorphic to $\Cl(X)$, by \cite[Example 1.2.3.4]{ArzhantsevDerenthalHausenLaface2015}, the characteristic quasitorus $\bT \coloneqq \Spec(\bK[\Cl(X)])$ acts naturally on $\mathcal{R}(\mathbb{P} T_X)$ such that
		\[
		\mathcal{R}(\mathbb{P} T_X)^{\bT} = \bigoplus_{p \geq 0} H^0(\mathbb{P} T_X, \sO_{\mathbb{P} T_X}(p)) \cong \bigoplus_{p \geq 0} H^0(X, \aS^p T_X) = S(X).
		\]
		It then follows from Lemma \ref{l.PE} and Boutot's theorem \cite{Boutot1987} that $S(X)$ has rational singularities, proving the statement \ref{i1.fg}.
		
		Finally, if $\Cl(X) = 0$, then the quasi-torus $\bT$ is trivial so that $S(X) = \mathcal{R}(\mathbb{P} E^*)$ is a UFD, implying that $\overline{T^*X}$ is factorial. If $\Cl(X)$ is finite, $\bT$ is also finite, then it follows from \cite[Lemma 5.16]{KollarMori1998} that $\overline{T^*X}$ is $\mathbb{Q}$-factorial. The statement \ref{i2.factorial} follows.
	\end{proof}
	
	We require the following criterion, which follows from \cite{Namikawa2001}.
	
	\begin{prop}
		\label{prop.symp=rat}
		Let $X$ be a quasi-affine smooth variety. Then $\overline{T^*X}$ is symplectic if and only if $S(X)$ has rational singularities.
	\end{prop}
	
	\begin{proof}
		The ``only if'' direction follows from \cite[Proposition 1.3]{Beauville2000a}. For the ``if'' direction, note that $T^*X$ is quasi-affine, so the natural morphism $T^*X\rightarrow \overline{T^*X}$ is an open embedding such that $\codim(\overline{T^*X}\setminus T^*X)\geq 2$ by Theorem \ref{Thm.Grosshans-SmallBoundary}. In particular, the canonical symplectic form $\omega$ on $T^*X$ extends to a symplectic form on $(\overline{T^*X})_{\reg}$, which implies that the canonical divisor of $\overline{T^*X}$ is trivial and hence Cartier. Then \cite[Corollaries 5.24 and 5.25]{KollarMori1998} show that $\overline{T^*X}$ is Cohen-Macaulay and hence has rational Gorenstein singularities. The result now follows from \cite[Theorem 6]{Namikawa2001}.
	\end{proof}
	
	\begin{proof}[Proof of Theorem \ref{t.MainThmHorospherical}]
		Since $G/H$ is a quasi-affine horospherical space, the result follows from Propositions \ref{prop.GeneralHoro} and \ref{prop.symp=rat}. 
	\end{proof}

	\begin{proof}[Proof of Corollary \ref{c.U_P/[P,P]-Cotangent}]
		Let $X=G/[P,P]$. Recall from \cite[Proposition 3.2]{KnopKraftVust1989} that there exists an exact sequence
		\[
		0=\fX([P,P]) \longrightarrow \Cl(X) \longrightarrow \Cl(G),
		\]
		where $\fX([P,P])$ is the character group of $[P,P]$. It is known that $\Cl(G)$ is always finite (see \cite[Theorem 4.2.2.1]{ArzhantsevDerenthalHausenLaface2015}), and $\Cl(G)$ is trivial if $[G,G]$ is simply connected by \cite[Corollary]{Popov2023}. Now the result follows from Theorem \ref{t.MainThmHorospherical} and Proposition \ref{prop.GeneralHoro}.
	\end{proof}

\section{Cones of highest weight vectors}
	
	\subsection{HV cones}
	\label{ss.HVcones}
	
	Let $G$ be a connected reductive group. Let $\chi$ be a dominant weight of $G$ with $V_\chi$ being the simple $G$-module with highest weight $\chi$. Let $v\in V_\chi^*$ be a highest weight vector in the dual of $V_\chi$. Recall that the Zariski closure $C_\chi\coloneqq \overline{G\cdot v}\subset V_\chi^*$ is called the \emph{cone of highest weight vectors (HV cone)}. 
	
	The projectivization of $C_\chi$ is a rational homogeneous space $Z\coloneqq G/P_\chi\subset \bP(V_\chi)$, where $P_\chi$ is the stabilizer of $[v]\in \bP(V_\chi)$; that is, $C_\chi$ is the affine cone $\widehat{Z}$ of $Z$. Moreover, it is known that the punctured cone 
	\[
	X\coloneqq \widehat{Z}\setminus \{0\}=C_\chi\setminus\{0\}
	\]
	is exactly $G\cdot v\subset V_\chi^*$, which is a horospherical space. Then Theorem \ref{t.MainThmHorospherical} implies that $\overline{T^*X}$ is a symplectic variety. In particular, it follows that $\overline{T^*\cO_{\min}}$ is a symplectic variety for the minimal nilpotent orbit $\cO_{\min}$ in a simple Lie algebra, and it seems plausible that this property also holds for all nilpotent orbits in a semi-simple Lie algebra.
	
	\subsection{Associated varieties}

	Let $G$ be a connected simple group, and let $P \subset G$ be a maximal parabolic subgroup containing a Borel subgroup $B$. Let $\mathfrak{h} \subset \mathfrak{b}$ be the Cartan subalgebra. Let $\mathfrak{g} = \mathfrak{n}^+ \oplus \mathfrak{h} \oplus \mathfrak{n}^-$ be the usual triangular decomposition relative to the choice of $\mathfrak{h}$ and $\mathfrak{b} = \mathfrak{h} \oplus \mathfrak{n}^+$. Write $\mathfrak{p} = \mathfrak{l} \oplus \mathfrak{u}^+$, where $\mathfrak{u}^+$ is the nilpotent radical of $\mathfrak{p}$ and $\mathfrak{l}$ is the Levi part. Thus, $\mathfrak{g} = \mathfrak{u}^+ \oplus \mathfrak{l} \oplus \mathfrak{u}^-$. We will consider certain cases where $\mathfrak{u}^+$ is abelian, summarized in the following table (see \cite[Table 3.1]{LevasseurSmithStafford1988} and \cite[Lemmas, p.~32 and p.~37]{LevasseurStafford1989}).
	
	{\tiny
		\begin{longtable}{llp{2.3cm}lll}
			\caption{}\label{table.Abelian}\\
			\hline
			\text{Type of $\mathfrak{g}$} 
			& \text{Dynkin diagram} 
			& parabolics $\mathfrak{p}$ with abelian radical
			& nilpotent orbit $\mathcal{O}$
			& $\dim(\mathcal{O})$
			& $\dim(X)$ \\
			\hline
			\endfirsthead
			\multicolumn{6}{l}{{ {\bf \tablename\ \thetable{}} \textrm{-- continued from previous page}}}
			\\
			\hline 
			\text{Type of $\mathfrak{g}$} 
			& \text{Dynkin diagram} 
			& parabolics $\mathfrak{p}$ with abelian radical
			& nilpotent orbit $\mathcal{O}$
			& $\dim(\mathcal{O})$
			& $\dim(X)$
			\endhead
			\hline
			\hline \multicolumn{6}{r}{{\textrm{Continued on next page}}} \\ \hline
			\endfoot
			
			\hline \hline
			\endlastfoot
			$A_n$ $(n \geq 1)$ 
			& $\dynkin[labels={1,2, ,n}] A{*2.*2}$
			& $\mathfrak{p}_i$ $\left(2 \leq i \leq \frac{n+1}{2}\right)$ 
			& $\mathcal{O}_r$ $(1 \leq r \leq i-1)$ 
			& $2r(n - r + 1)$
			& $r(n - r + 1)$\\
			$B_n$ $(n \geq 3)$ 
			& $\dynkin[labels={1,2,,n}] B{*2.*2}$
			& $\mathfrak{p}_1$ 
			& $\mathcal{O}_{\min}$ 
			& $4n - 4$ 
			& $2n - 2$ \\
			$C_n$ $(n \geq 2)$
			& $\dynkin[labels={1,2,,n}] C{*2.*2}$
			& $\mathfrak{p}_n$
			& $\mathcal{O}_r$ $(1 \leq r \leq n - 1)$ 
			& $r(2n - r + 1)$ 
			& $nr - \frac{r(r - 1)}{2}$ \\
			$D_n$ $(n \geq 4)$
			& $\dynkin[labels={1,2,,n-1,n}] D{*2.*3}$
			& $\mathfrak{p}_1$, $\mathfrak{p}_{n-1}$, $\mathfrak{p}_n$
			& $\mathcal{O}_{\min}$
			& $4n - 6$
			& $2n - 3$ \\
			$D_n$ $(n \geq 4)$
			& $\dynkin[labels={1,2,,n-1,n}] D{*2.*3}$
			& $\mathfrak{p}_n$
			& $\mathcal{O}_{r}$ $\left(2 \leq r \leq \frac{n - 2}{2}\right)$
			& $2r(2n - 2r - 1)$
			& $r(2n - 2r - 1)$ \\
			
			$E_6$ 
			& $\dynkin[labels={1,...,6}] E6$
			& $\mathfrak{p}_1$, $\mathfrak{p}_6$
			& $\mathcal{O}_{\min}$  
			& $22$
			& $11$ \\
			$E_7$
			& $\dynkin[labels={1,...,7}] E7$
			& $\mathfrak{p}_7$
			& $\mathcal{O}_{\min}$
			& $34$
			& $17$ \\
			\hline
		\end{longtable}
	}
	
	The extra notation in Table \ref{table.Abelian} is as follows. In each case, $\mathfrak{p}_i$ is the maximal parabolic subalgebra of $\mathfrak{g}$ associated to the simple root $\alpha_i$ (in the order of Bourbaki), and $\mathcal{O}_{\min}$ denotes the minimal non-zero nilpotent orbit in $\mathfrak{g}$. In Cases $A_n$, $C_n$, and $D_n$, the variety $\mathcal{O}_r$ is a nilpotent orbit in $\mathfrak{g}$ defined as follows (see \cite[II, \S\,5 and \S\,6]{LevasseurStafford1989}):
	\[
	\mathcal{O}_r \coloneqq \{z \in \mathfrak{g} \mid \rank(z) = \widetilde{r} \text{ and } z^2 = 0\},
	\]
	where $\widetilde{r} = r$ in Cases $A_n$ and $C_n$, but $\widetilde{r} = 2r$ in Cases $D_n$. Once both $\mathfrak{p}$ and $\mathcal{O}$ are fixed, we define $X$ to be $\mathcal{O} \cap \mathfrak{u}^+$.
	
	Now consider the $\mathbb{G}_m$-action on $\mathfrak{g}$ with weight $0$ (resp. $1$ and $2$) on $\mathfrak{u}^+$ (resp. $\mathfrak{l}$ and $\mathfrak{u}^-$), namely $\phi: \mathbb{G}_m \times \mathfrak{g} \to \mathfrak{g}$ such that
	\[
	\phi_\lambda(u^+ + l + u^-) = u^+ + \lambda l + \lambda^2 u^-.
	\]
	
	\begin{lem} \label{l.Lie}
		For any $v, x, y \in \mathfrak{g}$, we have:
		\begin{enumerate}
			\item \label{l.Lie1} $[\phi_\lambda(x), \phi_\lambda(y)] = \lambda \phi_\lambda([x, y])$.
			
			\item \label{l.Lie2} $\kappa(\phi_\lambda(v), [\phi_\lambda(x), \phi_\lambda(y)]) = \lambda^3 \kappa(v, [x, y])$, where $\kappa$ is the Killing form.
			
			\item \label{l.Lie3} Any nilpotent orbit in $\mathfrak{g}$ is stable under this $\mathbb{G}_m$-action.
		\end{enumerate}
	\end{lem}
	
	\begin{proof}
		Write $x = x^+ + x_0 + x^-$ (resp. $y = y^+ + y_0 + y^-$) in the decomposition of $\mathfrak{g} = \mathfrak{u}^+ \oplus \mathfrak{l} \oplus \mathfrak{u}^-$. Recall that $\mathfrak{u}^+$ and $\mathfrak{u}^-$ are both abelian, hence we have
		\begin{align*}
			[\phi_\lambda(x), \phi_\lambda(y)] &= [x^+ + \lambda x_0 + \lambda^2 x^-, y^+ + \lambda y_0 + \lambda^2 y^-] \\
			&= \lambda([x^+, y_0] + [x_0, y^+]) + \lambda^2([x^+, y^-] + [x_0, y_0] + [x^-, y^+]) \\
			&\quad + \lambda^3([x_0, y^-] + [x^-, y_0]) \\
			&= \lambda \phi_\lambda([x, y]).
		\end{align*}
		
		Now consider the map $\psi_\lambda = \frac{1}{\lambda} \phi_\lambda$. By \ref{l.Lie1}, $\psi_\lambda$ preserves the Lie bracket. This implies that $\psi_\lambda$ preserves the Killing form, hence we have
		\begin{align*}
			\kappa(\phi_\lambda(v), [\phi_\lambda(x), \phi_\lambda(y)]) &= \kappa(\phi_\lambda(v), \lambda \phi_\lambda([x, y])) \\
			&= \kappa(\lambda \psi_\lambda(v), \lambda^2 \psi_\lambda([x, y])) \\
			&= \lambda^3 \kappa(v, [x, y]).
		\end{align*}
		
		Since $\psi_\lambda$ preserves the Lie bracket, its action on $\mathfrak{g}$ is via the adjoint action, hence it preserves any adjoint orbits in $\mathfrak{g}$. As nilpotent orbits are stable under the dilation action of scalars, it follows that $\phi_\lambda$ preserves nilpotent orbits.
	\end{proof}

	\subsection{Proof of Theorem \ref{thm.coneIHSS}}
	Let $\mathfrak{g}$ be a simple Lie algebra with $\mathfrak{p}$, $\mathcal{O}$, and $X$ as in Table \ref{table.Abelian}. We recall that the affine closure of $X$ (resp. $\mathcal{O}$) is exactly the Zariski closure of $X$ (resp. $\mathcal{O}$) in $\mathfrak{u}^+$ (resp. $\mathfrak{g}$) according to \cite[Remarks 3.2]{LevasseurSmithStafford1988} and \cite[Lemmas, p.~32 and p.~37]{LevasseurStafford1989}. In particular, the Zariski closures of $X$ and $\0$ are both normal.

	\begin{thm}
		\label{thm.asscociatedvar}
		The affine closure $\overline{T^* X}$ is isomorphic to $\overline{\mathcal{O}} \subset \mathfrak{g}$.
	\end{thm}
	
	\begin{proof}
		For an element $v \in \mathcal{O}$, the tangent space $T_v \mathcal{O}$ is naturally identified with $[v, \mathfrak{g}]$. By definition, the symplectic form on $\mathcal{O}$ is given by 
		\[
		\omega_v([v, x], [v, y]) = \kappa(v, [x, y]).
		\]
		
		We now show that $\phi_\lambda^* \omega = \lambda \omega$ for all $\lambda \in \mathbb{G}_m$. Take any $v \in \mathcal{O}$ and $x, y \in \mathfrak{g}$; then by Lemma \ref{l.Lie}, we have
		\begin{align*}
			(\phi_\lambda^* \omega)_v([v, x], [v, y]) &= \omega_{\phi_\lambda(v)}(d \phi_\lambda([v, x]), d \phi_\lambda([v, y])) \\
			&= \omega_{\phi_\lambda(v)}\left([\phi_\lambda(v), \frac{1}{\lambda} \phi_\lambda(x)], [\phi_\lambda(v), \frac{1}{\lambda} \phi_\lambda(y)]\right) \\
			&= \kappa\left(\phi_\lambda(v), \frac{1}{\lambda^2} [\phi_\lambda(x), \phi_\lambda(y)]\right) \\
			&= \frac{1}{\lambda^2} \kappa(\phi_\lambda(v), [\phi_\lambda(x), \phi_\lambda(y)]) \\
			&= \lambda \kappa(v, [x, y]) \\
			&= \lambda \omega_v([v, x], [v, y]).
		\end{align*}
		
		Notice that $\dim X = \frac{1}{2} \dim \mathcal{O}$, and the $\mathbb{G}_m$-fixed locus is $\mathcal{O} \cap \mathfrak{u}^+ = X$. It follows that $X$ is a Lagrangian submanifold of $\mathcal{O}$. 
		For any $x \in X$, let $T_x^p \coloneqq \{v \in T_x \mathcal{O} \mid (\phi_\lambda)_*(v) = \lambda^p v\}$. In particular, $T_x^0 = T_x X$. Take any $v_1 \in T_x^p$ and $v_2 \in T_x^q$. Since $\phi_\lambda^* \omega = \lambda \omega$, we have
		\[
		\lambda \omega(v_1, v_2) = \phi_\lambda^* \omega(v_1, v_2) = \omega((\phi_\lambda)_* v_1, (\phi_\lambda)_* v_2) = \lambda^{p + q} \omega(v_1, v_2).
		\]
		Hence, if $p + q \neq 1$, then $\omega(v_1, v_2) = 0$. This shows that $T_x^0 \simeq T_x^1$ and $T_x^p = 0$ for $p \neq 0, 1$ as $\dim T_x^0 = \dim X =  \frac{1}{2} \dim \mathcal{O}$. Then we get a weight decomposition as follows:
		\begin{equation}
			\label{eq.Weight-Decomp}
			T {\mathcal{O}}|_X = T^0\oplus T^1 \cong  TX \oplus N_{X/\mathcal{O}}.
		\end{equation}
		Consider the map $\pi_+ \colon \mathfrak{g} \to \mathfrak{u}^+$, the linear projection along $\mathfrak{l} + \mathfrak{u}^-$, which is $\mathbb{G}_m$-equivariant sending a point $v$ to its $\mathbb{G}_m$-limit $\lim_{\lambda \to 0} \phi_\lambda(v)$. Then the image $\pi_+(\mathcal{O})$ is the closure $\overline{X}$ of $X$ in $\fu^+$. Denote by $W$ the pre-image $\pi_+^{-1}(X)$ in $\mathcal{O}$, which consists of points $v \in \mathcal{O}$ such that $\lim_{\lambda \to 0} \phi_\lambda(v) \in X$; so $W$ is a dense open subset of $\mathcal{O}$. 
		
		Now we apply the argument of \cite[Lemma 3.7]{Fu2003} to show that $W$ is isomorphic to $T^*X$. By the Bia{\l}ynicki-Birula decomposition \cite{BialynickiBirula1973}, the projection 
		\[
		\pi_+|_W\colon W \longrightarrow X
		\]
		is a vector bundle of rank $\dim X$, which is canonically isomorphic to the normal bundle of its zero section, thus $W$ is natually isomorphic to the normal bundle $N_{X/\cO}$ of $X$ in $\mathcal{O}$ because $W$ is an open subset of $\mathcal{O}$ containing $X$. On the other hand, since $X$ is Lagrangian, the symplectic form on $\mathcal{O}$ induces an isomorphism $T^*X\cong N_{X/\mathcal{O}}$ via the decomposition \eqref{eq.Weight-Decomp}, which then yields an isomorphism $W\cong T^*X$ as vector bundles.     
	\end{proof}
	
	\begin{rem}
		Pick $\mathfrak{g}$, $X$, and $\mathcal{O}$ in Table \ref{table.Abelian}. Let $U(\mathfrak{g})$ be the universal enveloping algebra of $\mathfrak{g}$, and let $\sD(X)$ be the algebra of differential operators on $X$. By \cite[Theorem 5.2]{LevasseurSmithStafford1988} and \cite[Theorem, p.~2]{LevasseurStafford1989}, there exists a completely prime, maximal ideal $J$ (depending on $\mathcal{O}$) of $U(\mathfrak{g})$ and an isomorphism $\psi \colon U(\mathfrak{g})/J \rightarrow \sD(X)$. The Poincar\'e--Birkhoff--Witt filtration on $U(\mathfrak{g})$ gives a graded ring:
		\[
		R \coloneqq \gr U(\mathfrak{g}) / \gr J = \aS^{\bullet} \mathfrak{g} / \gr J.
		\]
		Then the induced closed embedding $\Spec(R) \subset \mathfrak{g}^*$ is nothing but $\overline{\mathcal{O}} \subset \mathfrak{g}$ under the natural isomorphism $\mathfrak{g} \cong \mathfrak{g}^*$ induced by $\kappa$. On the other hand, the natural order filtration on $\sD(X)$ induces a graded morphism $\gr \sD(X) \rightarrow S(X) = k[T^*X]$. Then Theorem \ref{thm.asscociatedvar} implies that there exists a natural isomorphism $R \rightarrow S(X)$. This is a surprise to us because $\psi$ does not preserve the filtrations on each side (cf. \cite[IV, \S\,1.9]{LevasseurStafford1989}).
	\end{rem}

    Let $S$ be an irreducible Hermitian symmetric space (IHSS), and let $Z\subset \bP T^*_o S$ be the variety of lines in $X$ passing through a reference point $o\in S$. Then $Z$ is also a Hermitian symmetric space and the inclusion $Z\subset \bP T^*_o S$ is an equivariant embedding. In the following table, we collect the classification of IHSSs $S$, different from projective spaces, and the associated varieties $Z$, together with a description of the inclusion $Z\subset \bP T_o^*S$. Here in this table $\bQ^n$ stands for the $n$-dimensional smooth hyperquadric and $G/P_k$ stands for the rational homogeneous space associated to the maximal parabolic subgroup determined by the $k$-th simple root (in the ordering of Bourbaki) (see \cite[(2.1)]{HwangMok1998}).
    
    	\begin{longtable}{ccccccc}
    		\caption{}\label{Table.IHSS-Cone}
    		\\
    		\hline  
    		IHSS  & {\rm Gr}(a, a+b) & $D_n/P_n$   & $C_n/P_n$ &    $\mathbb{Q}^n$ & $E_6/P_1$ & $E_7/P_7 $ \\
    		\hline  
    		$Z$   &  $\bP^{a-1} \times \bP^{b-1}$ &     ${\rm Gr}(2, n)$    &$\bP^{n-1}$  & $\mathbb{Q}^{n-2}$ & $D_5/P_5$ & $E_6/P_1$  \\
    		\hline  
    		$Z \subset \bP T_o^* S$ & Segre & Pl\"ucker &$2^{\text{nd}}$ Veronese  & Hyperquadric  & Spinor  & Severi\\
    		\hline
    	\end{longtable}
	
	\begin{proof}[Proof of Theorem \ref{thm.coneIHSS}]
		For an IHSS $G/P$ different from projective spaces, we have the corresponding decomposition $\mathfrak{g} = \mathfrak{u}^+ \oplus \mathfrak{l} \oplus \mathfrak{u}^-$. The tangent space $T_{o}(G/P)$ at a base point $o$ is naturally identified with $\mathfrak{u}^-$.
		Then $Z\subset \bP T_o^*(G/P)$ is the highest weight variety in $\bP(\fu^-)^* \cong \bP (\mathfrak{u}^+)^*$, whose punctured cone is just the variety $X$ (see \cite[(5.1)]{HwangMok1998}). Now the result follows from Theorem \ref{thm.asscociatedvar}.
	\end{proof}

	\begin{rem}
		In Theorem \ref{thm.coneIHSS}, if $G/P$ is a projective space, i.e., $(\mathfrak{g}, \mathfrak{p}) = (\mathfrak{sl}(n+1), \mathfrak{p}_1)$, the argument in Theorem \ref{thm.asscociatedvar} remains valid. However, as $\dim(\mathcal{O} \setminus W) = \dim(\mathcal{O}) - 1$ at the final step (see \cite[Proposition 4.8]{LevasseurSmithStafford1988}), the natural inclusion $\bK[\mathcal{O}]\subset \bK[W]=\bK[T^*X]$ is strict.
	\end{rem}
	
	\begin{example} \label{e.quadric}
		Consider the $n$-dimensional hyperquadric cone $\widehat{\bQ}^n \subset \bA^{n+1}$ over a smooth projective hyperquadric in $\mathbb{P}^n$, $n\geq 2$. Let $X = \widehat{\bQ}^n \setminus \{0\}$; then $\overline{T^*X}$ is a symplectic variety of dimension $2n$, which is isomorphic to the minimal nilpotent orbit closure $\overline{\mathcal{O}}_{\min}$ in $\mathfrak{so}_{n+3}$ by Theorem \ref{thm.coneIHSS}. This generalizes a result proven in \cite{Jia2025a} for $n=5$ by another method (cf. \cite{Jia2025}).
	\end{example}

\section{Parabolic basic affine spaces}

In this section, we apply the idea of \S\,\ref{s.Horospherical} to investigate the affine closure of the cotangent bundles of parabolic basic affine spaces, proving Theorem \ref{t.G/UP} and providing some evidence for Conjecture \ref{conj.PEFanoType}.

Let $G$ be a connected reduction group. Then there exists a simply connected semisimple group $G^{\ssc}$, a torus $A$ and a closed finite central subgroup $K$ such that $G$ can be written as a quotient 
\[
(G^{\ssc}\times A)/K \stackrel{\cong}{\longrightarrow} G.
\]

Let $P$ be a parabolic subgroup of $G$ with $U_P$ its unipotent radical. Then the parabolic basic affine space $G/U_P$ is quasi-affine. Moreover, let $U_{P^{\ssc}}$ be the unipotent radical of a parabolic subgroup $P^{\ssc}$ of $G^{\ssc}$ such that $P^{\ssc}\times A$ projects into $P$ under the quotient map. Then $G^{\ssc}/U_{P^{\ssc}}\times A\rightarrow G/U_{P}$ is an \'etale Galois cover so that the induced morphism
\begin{equation}
\label{eq.SC-Reduction}
    T^*(G^{\ssc}/U_{P^{\ssc}}) \times T^*A \longrightarrow T^*(G/U_{P})
\end{equation}
is again an \'etale Galois cover. In this case, we have the following:

\begin{lem}
	\label{lem.sc}
	\begin{enumerate}
		\item\label{i1.lem.sc} The $\bK$-algebra $S(G/U_P)$ is finitely generated if and only if the $\bK$-algebra $S(G^{\ssc}/U_{P^{\ssc}})$ is finitely generated.
		
		\item\label{i2.lem.sc} The affine closure $\overline{T^*(G/U_P)}$ has symplectic singularities if and only if $\overline{T^*(G^{\ssc}/U_{P^{\ssc}})}$ has symplectic singularities.
	\end{enumerate}
\end{lem}

\begin{proof}
	The statement \ref{i1.lem.sc} follows from the \'etale Galois cover \eqref{eq.SC-Reduction}, Proposition \ref{prop.FiniteCover} and the K\"unneth formula (cf. Proof of Proposition \ref{prop.GeneralHoro}).

	For the statement \ref{i2.lem.sc} we can assume that $S(G^{\ssc}/U_{P^{\ssc}})$ and $S(G/U_P)$ are finitely generated by \ref{i1.lem.sc}, then the induced morphism
	\[
	W\coloneqq \overline{T^*(G^{\ssc}/U_{P^{\ssc}})}\times T^*A \cong \overline{T^*(G^{\ssc}/U_{P^{\ssc}})\times T^*A} \longrightarrow \overline{T^*(G/U_P)}\eqqcolon Z
	\]
	is a finite Galois cover by Proposition \ref{prop.FiniteCover}. In particular, according to Proposition \ref{prop.symp=rat}, it suffices to prove that $Z$ has rational singularities if and only if $\overline{T^*(G^{\ssc}/U_{P^{\ssc}})}$ has rational singularities. 
	
	The ``if'' direction follows from Boutot's theorem \cite{Boutot1987}. For the ``only if'' direction, suppose that $Z$ is symplectic. Thus $Z$ has canonical singularities by \cite[Corollary 5.24]{KollarMori1998} and then \cite[Proposition 5.20]{KollarMori1998} implies that $(W,\Delta)$ has klt singularities for some effective $\bQ$-divisor $\Delta$ on $W$. Consequently, \cite[Theorem 5.22]{KollarMori1998} implies that $W$ has rational singularities and therefore $\overline{T^*(G^{\ssc}/U_{P^{\ssc}})}$ also has rational singularities. 
\end{proof}

Let $B$ be the Borel subgroup of $G$ contained in $P$ and let $\fu_P^{\perp}\subset \fg^*$ be the annihilator of $\fu_P$. Then the cotangent bundle $T^*(G/U_P)$ can be naturally identified to the geometric quotient
\[
(G\times \fu_P^{\perp})/U_P\coloneqq G\times_{U_P} \fu_P^{\perp}
\]
where $U_P$ acts on $G$ by right multiplication and on $\fu_P^{\perp}$ by the coadjoint action. The ring $S(G/U_P)$ is isomorphic to the $U_P$-invariant subring $\bK[G \times \fu_P^{\perp}]^{U_P}$, and $\overline{T^*(G/U_P)}$ is isomorphic to the invariant-theoretic quotient $(G \times \fu_P^{\perp}) /\!\!/ U_P$. We require the following result.

\begin{prop}[\protect{\cite[Theorem 16.2]{Grosshans1997}}]
	\label{t.QuobytU_P}
	Let $G$ be a connected reductive group, and let $U_P$ be the unipotent radical of a parabolic subgroup $P$ of $G$. Let $R$ be a $\bK$-algebra on which $G$ acts rationally. Then $R$ is finitely generated (resp. integrally closed, or has rational singularities) if and only if $R^{U_P}$ has the same property.
\end{prop}

Let $M$ be the homogeneous bundle $G \times_{U_B} \fu_P^{\perp}$ over $G / U_B$.

   \begin{lem}
   	\label{lem.S(X_P)=k[M]}
   	The $\bK$-algebra $S(G/U_P)$ is finitely generated (resp. has rational singularities) if and only if $\bK[M]$ is so.
   \end{lem}
   
   \begin{proof}
   	Let $P = U_P L_P$ be the Levi decomposition. Since $P$ is contained in the normalizer of $U_P$ in $G$, the natural $U_P$-action on $G \times \fu_P^{\perp}$ extends to a $P$-action, and thus the reductive group $L_P$ acts naturally on $G \times \fu_P^{\perp}$. Denote by $U_B$ the unipotent radical of $B$ and by $U_L$ a unipotent subgroup of $L_P$ such that $U_B = U_P \rtimes U_L$. Then we have
   	\[
   	S(G/U_P)^{U_L} \cong (\bK[G \times \fu_P^{\perp}]^{U_P})^{U_L} = \bK[G \times \fu_P^{\perp}]^{U_B},
   	\]
   	where the $U_L$-action on $S(G/U_P)$ is induced by the natural $L_P$-action on $T^*(G/U_P)$. Notice that $(G \times \fu_P^{\perp}) /\!\!/ U_B$ is just the affine closure $\overline{M}$ of $M$. Now applying Proposition \ref{t.QuobytU_P} to $L_P$ and $S(G/U_P)$ yields the desired result.
   \end{proof}
   
   Let $E$ be the homogeneous vector bundle $G \times_B \fu_P^{\perp}$ over $G/B$.
   
   \begin{lem}
   	\label{lem.CoxPE=CoxPM}
   	There exists a graded isomorphism $\mathcal{R}(\mathbb{P} E^*) \cong \mathcal{R} (\mathbb{P} M^*)$.
   \end{lem}
   
   \begin{proof}
   	Let $\mathbb{T}$ be the torus $B / U_B$. Then the natural projection
   	\[
   	M \coloneqq  G \times_{U_B} \fu_P^{\perp} \longrightarrow G \times_B \fu_P^{\perp}\eqqcolon E
   	\]
   	is exactly the geometric quotient of $M$ by the natural induced $\mathbb{T}$-action on $M$. Let $\mathbb{P} M^*$ and $\mathbb{P} E^*$ be the projectivizations of dual bundles. The induced morphism $q \colon \mathbb{P} M^* \rightarrow \mathbb{P} E^*$ is the geometric quotient by the induced $\mathbb{T}$-action on $\mathbb{P} M^*$. The statement can then be derived immediately from Theorem \ref{t.Coxring-torusquotient}.
   \end{proof}
   
   \begin{proof}[Proof of Theorem \ref{t.G/UP}]
   	By Lemma \ref{lem.sc}, we may assume that $G$ is simply connected and semisimple so that $\Cl(G)$ is trivial, then \cite[Proposition 3.2]{KnopKraftVust1989} implies that $\Cl(G/U_P)$ is trivial as $\fX(U_P)=0$. Now we have the following isomorphisms:
   	\begin{align*}
   		\bigoplus_{p \geq 0} H^0(\mathbb{P} M^*,\sO_{\mathbb{P} M^*}(p)) 
   		& \cong \mathcal{R}(\mathbb{P} M^*)
   		& \text{since } \operatorname{Cl}(\mathbb{P} M^*) \cong \mathbb{Z} \sO_{\mathbb{P} M^*}(1) \\
   		& \cong \mathcal{R}(\mathbb{P} E^*)
   		& \text{by Lemma } \ref{lem.CoxPE=CoxPM}.
   	\end{align*}
   	In particular, combining Lemma \ref{lem.S(X_P)=k[M]} with the following isomorphism
   	\[
   	\bK[M] \cong \bigoplus_{p \geq 0} H^0(G / U_B, \aS^p M^*) \cong \bigoplus_{p \geq 0} H^0(\mathbb{P} M^*,\sO_{\mathbb{P} M^*}(p)).
   	\]
   	shows that $S(G/U_P)$ is finitely generated (resp. has rational singularities) if and only if $\cR(\bP E^*)$ is finitely generated (resp. has rational singularities), therefore the statement \ref{i1.t.G/UP} follows by Definition \ref{defn.MDS-Fano}.
   	
   	For the statement \ref{i2.t.G/UP}, it remains to prove that $\cR(\bP E^*)$ has rational singularities if and only if $\bP E^*$ is of Fano type. The ``if'' direction follows from \cite[Theorem 1.1]{GongyoOkawaSannaiTakagi2015} and \cite[Theorem 5.22]{KollarMori1998}. For the ``only if'' part, suppose that $\cR(\bP E^*)$ has rational singularities. Notice that $\Cl(\bP E^*)$ is a finitely generated free abelian group so that $\cR(\bP E^*)$ is a UFD by \cite[Corollary 1.2]{ElizondoKuranoWatanabe2004}. Then it follows from \cite[Corollary 5.24]{KollarMori1998} that $\Spec (\cR(\bP E^*))$ has canonical singularities and hence $\bP E^*$ is of Fano type by \cite[Theorem 1.1]{GongyoOkawaSannaiTakagi2015}.
   \end{proof}

   Let $F = G \times_P \fu_P^{\perp}$. Then there exists a natural fibration $\mathbb{P} E^* \rightarrow \mathbb{P} F^*$. In particular, if $\mathbb{P} E^*$ is of Fano type, then $\mathbb{P} F^*$ is also of Fano type by \cite[Corollary 1.3]{GongyoOkawaSannaiTakagi2015}. As evidence for Conjecture \ref{conj.PEFanoType}, we confirm this expectation in the following.
   
   \begin{prop} \label{p.weakfano}
   	The projectivization $\mathbb{P} F^*$ is a weak Fano manifold.
   \end{prop}
   
   \begin{proof}
   	Consider the following exact sequence of $P$-modules
   	\[
   	0 \longrightarrow \fu_P \longrightarrow \fg \longrightarrow \fg / \fu_P \longrightarrow 0.
   	\]
   	Note that $\fg / \fu_P \cong (\fu_P^\perp)^*$ and $\fu_P \cong (\fg / \fp)^*$, so we have the following exact sequence of vector bundles over $G / P$:
   	\begin{equation} \label{e.exact}
   		0 \to T^*_{G/P} \to G/P \times \fg \to F^* \to 0.
   	\end{equation}
   	Being a quotient of a trivial bundle, the vector bundle $F^*$ is globally generated. Moreover, we have $c_1(F^*) = -c_1(T^*(G / P)) = -K_{G / P}$.
   	
   	Let $p \colon \mathbb{P} F^* \to G / P$ be the natural projection. Using the relative Euler sequence, we have $c_1(T_p) = p^* c_1(F) + r \zeta$, where $r$ is the rank of $F$ and $\zeta = c_1(\mathcal{O}_{\mathbb{P} F^*}(1))$.
   	\[
   	-K_{\mathbb{P} F^*} = c_1(T_p) - p^* K_{G / P} = p^* c_1(F) + r \zeta - p^* K_{G / P} = r \zeta.
   	\]
   	Hence $\mathbb{P} F^*$ is weak Fano if and only if $\zeta$ is nef and big. Moreover, since $F^*$ is globally generated, the divisor $\zeta$ is a nef divisor. Hence, to show $\zeta$ is big, we only need to compute its top self-intersection, or equivalently, the top Segre class of $F$.
   	
   	By \cite[Proposition 10.3]{EisenbudHarris2016} and the exact sequence \eqref{e.exact}, we have
   	\[
   	s(F) = \frac{1}{c(F)} = c(T_{G / P}),
   	\]
   	where $s(F)$ is the total Segre class of $F$ and $c(T_{G / P})$ is the total Chern class of $T_{G / P}$. It follows that $s_{\text{top}} (F) = c_{\text{top}}(G / P)$, which is just the topological Euler number of $G / P$, and therefore strictly positive.
   \end{proof}

\begin{example}
	Let $G = \mathrm{SL}_{r+1}$ and let $P$ be the parabolic subgroup which is the stabilizer of a $r$-dimensional linear subspace in $\bK^{r+1}$. Then, according to \cite[Example 3]{ArzhantsevTimashev2005}, we have:
	\[
	G/U_P \cong \{f \in \mathrm{Hom}(\bK^r, \bK^{r+1}) \mid f \text{ is injective}\} \subset \mathrm{Hom}(\bK^r, \bK^{r+1}).
	\]
	Note that the complement has codimension two, hence $\overline{G/U_P} = \mathrm{Hom}(\bK^r, \bK^{r+1})$. It follows that $T^*(G/U_P) \subset T^*\mathrm{Hom}(\bK^r, \bK^{r+1})$ has a complement of codimension two, from which we deduce  $\overline{T^*(G/U_P)} = T^*\mathrm{Hom}(\bK^r, \bK^{r+1})$.
\end{example}

\begin{example}
	Let $G = \mathrm{Sp}_4$ and $P$ be the stabilizer of a Lagrangian plane in $\bK^4$. By \cite[Example 4]{ArzhantsevTimashev2005}, the affine closure of $G/U_P$ is isomorphic to a $7$-dimensional hyperquadric cone $\widehat{\bQ}^7 \subset \bA^8$. Since $G/U_P \subset \widehat{\bQ}^7$ has a small boundary, we have 
    \[
    \bK[T^*(G/U_P)] = \bK[T^*(\widehat{\bQ}^7 \setminus \{0\})],
    \]
    which implies that $\overline{T^*(G/U_P)}$ is isomorphic to the minimal nilpotent orbit closure in $\mathfrak{so}_{10}$ by Example \ref{e.quadric}.
\end{example}
	
	\bibliographystyle{alpha}
	\bibliography{UHKI}

\newcommand{\etalchar}[1]{$^{#1}$}
\begin{thebibliography}{BCHM10}

\bibitem[ADHL15]{ArzhantsevDerenthalHausenLaface2015}
Ivan Arzhantsev, Ulrich Derenthal, J\"{u}rgen Hausen, and Antonio Laface.
\newblock {\em Cox rings}, volume 144 of {\em Cambridge Studies in Advanced
  Mathematics}.
\newblock Cambridge University Press, Cambridge, 2015.

\bibitem[AT05]{ArzhantsevTimashev2005}
Ivan~V. Arzhantsev and Dmitri~A. Timashev.
\newblock On the canonical embeddings of certain homogeneous spaces.
\newblock In {\em Lie groups and invariant theory}, volume 213 of {\em Amer.
  Math. Soc. Transl. Ser. 2}, pages 63--83. Amer. Math. Soc., Providence, RI,
  2005.

\bibitem[BB73]{BialynickiBirula1973}
Andrzej~S. Bia{\l}ynicki-Birula.
\newblock Some theorems on actions of algebraic groups.
\newblock {\em Ann. of Math. (2)}, 98:480--497, 1973.

\bibitem[BCHM10]{BirkarCasciniHaconMcKernan2010}
Caucher Birkar, Paolo Cascini, Christopher~D. Hacon, and James McKernan.
\newblock Existence of minimal models for varieties of log general type.
\newblock {\em J. Amer. Math. Soc.}, 23(2):405--468, 2010.

\bibitem[BDG{\etalchar{+}}22]{BourgetDancerGrimmingerHananyZhong2022}
Antoine Bourget, Andrew Dancer, Julius~F. Grimminger, Amihay Hanany, and
  Zhenghao Zhong.
\newblock Partial implosions and quivers.
\newblock {\em J. High Energy Phys.}, (7):Paper No. 49, 20, 2022.

\bibitem[Bea00]{Beauville2000a}
Arnaud Beauville.
\newblock Symplectic singularities.
\newblock {\em Invent. Math.}, 139(3):541--549, 2000.

\bibitem[BL24]{BeauvilleLiu2024}
Arnaud Beauville and Jie Liu.
\newblock The algebra of symmetric tensors on smooth projective varieties.
\newblock {\em Sci. China Math.}, to appear, 2024.

\bibitem[Bou87]{Boutot1987}
Jean-Fran\c{c}ois Boutot.
\newblock Singularit\'es rationnelles et quotients par les groupes r\'eductifs.
\newblock {\em Invent. Math.}, 88(1):65--68, 1987.

\bibitem[DHK21]{DancerHananyKirwan2021}
Andrew Dancer, Amihay Hanany, and Frances Kirwan.
\newblock Symplectic duality and implosions.
\newblock {\em Adv. Theor. Math. Phys.}, 25(6):1367--1387, 2021.

\bibitem[DKS13]{DancerKirwanSwann2013}
Andrew Dancer, Frances Kirwan, and Andrew Swann.
\newblock Implosion for hyperk{\"a}hler manifolds.
\newblock {\em Compos. Math.}, 149(9):1592--1630, 2013.

\bibitem[EH16]{EisenbudHarris2016}
David {Eisenbud} and Joe {Harris}.
\newblock {\em {3264 and all that. A second course in algebraic geometry.}}
\newblock Cambridge: Cambridge University Press, 2016.

\bibitem[EKW04]{ElizondoKuranoWatanabe2004}
E.~Javier Elizondo, Kazuhiko Kurano, and Kei-ichi Watanabe.
\newblock The total coordinate ring of a normal projective variety.
\newblock {\em J. Algebra}, 276(2):625--637, 2004.

\bibitem[FL25]{FuLiu2025}
Baohua Fu and Jie Liu.
\newblock Symplectic singularities arising from algebras of symmetric tensors.
\newblock {\em J. Math. Pures Appl. (9)}, 204:Paper No. 103794, 2025.

\bibitem[Fu03]{Fu2003}
Baohua Fu.
\newblock Symplectic resolutions for nilpotent orbits.
\newblock {\em Invent. Math.}, 151(1):167--186, 2003.

\bibitem[Gan24]{Gannon2024}
Tom Gannon.
\newblock Proof of the {G}inzburg-{K}azhdan conjecture.
\newblock {\em Adv. Math.}, 448:Paper No. 109701, 28, 2024.

\bibitem[GK22]{GinzburgKazhdan2022}
Victor Ginzburg and David Kazhdan.
\newblock Differential operators on {$G/U$} and the {G}elfand-{G}raev action.
\newblock {\em Adv. Math.}, 403:Paper No. 108368, 48, 2022.

\bibitem[GOST15]{GongyoOkawaSannaiTakagi2015}
Yoshinori Gongyo, Shinnosuke Okawa, Akiyoshi Sannai, and Shunsuke Takagi.
\newblock Characterization of varieties of {F}ano type via singularities of
  {C}ox rings.
\newblock {\em J. Algebraic Geom.}, 24(1):159--182, 2015.

\bibitem[GR15]{GinzburgRiche2015}
Victor Ginzburg and Simon Riche.
\newblock Differential operators on {$G/U$} and the affine {G}rassmannian.
\newblock {\em J. Inst. Math. Jussieu}, 14(3):493--575, 2015.

\bibitem[Gro97]{Grosshans1997}
Frank~D. Grosshans.
\newblock {\em Algebraic homogeneous spaces and invariant theory}, volume 1673
  of {\em Lecture Notes in Mathematics}.
\newblock Springer-Verlag, Berlin, 1997.

\bibitem[GW25]{GannonWebster2025}
Tom Gannon and Ben Webster.
\newblock Functoriality of {C}oulomb branches.
\newblock {\em arXiv preprint arXiv:2501.09962}, 2025.

\bibitem[Har77]{Hartshorne1977}
Robin Hartshorne.
\newblock {\em Algebraic geometry}.
\newblock Springer-Verlag, New York-Heidelberg, 1977.
\newblock Graduate Texts in Mathematics, No. 52.

\bibitem[Har80]{Hartshorne1980}
Robin Hartshorne.
\newblock Stable reflexive sheaves.
\newblock {\em Math. Ann.}, 254(2):121--176, 1980.

\bibitem[HM98]{HwangMok1998}
Jun-Muk Hwang and Ngaiming Mok.
\newblock Rigidity of irreducible {H}ermitian symmetric spaces of the compact
  type under {K}\"ahler deformation.
\newblock {\em Invent. Math.}, 131(2):393--418, 1998.

\bibitem[Jia25a]{Jia2025a}
Boming Jia.
\newblock The affine closure of {$T^*(\text{SL}_n/U)$}.
\newblock {\em J. Lie Theory}, 35(1):83--100, 2025.

\bibitem[Jia25b]{Jia2025}
Boming Jia.
\newblock Minimal nilpotent orbits of type ${D}$ and ${E}$.
\newblock {\em arXiv preprint arXiv:2501.12406}, 2025.

\bibitem[KKV89]{KnopKraftVust1989}
Friedrich Knop, Hanspeter Kraft, and Thierry Vust.
\newblock The {P}icard group of a {$G$}-variety.
\newblock In {\em Algebraische {T}ransformationsgruppen und
  {I}nvariantentheorie}, volume~13 of {\em DMV Sem.}, pages 77--87.
  Birkh\"auser, Basel, 1989.

\bibitem[KM98]{KollarMori1998}
J{\'a}nos Koll{\'a}r and Shigefumi Mori.
\newblock {\em Birational geometry of algebraic varieties}, volume 134 of {\em
  Cambridge Tracts in Mathematics}.
\newblock Cambridge University Press, Cambridge, 1998.

\bibitem[LS89]{LevasseurStafford1989}
Thierry Levasseur and J.~Toby Stafford.
\newblock Rings of differential operators on classical rings of invariants.
\newblock {\em Mem. Amer. Math. Soc.}, 81(412):vi+117, 1989.

\bibitem[LSS88]{LevasseurSmithStafford1988}
Thierry Levasseur, S.~Paul Smith, and J.~Toby Stafford.
\newblock The minimal nilpotent orbit, the {J}oseph ideal, and differential
  operators.
\newblock {\em J. Algebra}, 116(2):480--501, 1988.

\bibitem[Nam01]{Namikawa2001}
Yoshinori Namikawa.
\newblock Extension of 2-forms and symplectic varieties.
\newblock {\em J. Reine Angew. Math.}, 539:123--147, 2001.

\bibitem[Pop23]{Popov2023}
Vladimir~Leonidovich Popov.
\newblock Picard group of a connected affine algebraic group.
\newblock {\em Uspekhi Mat. Nauk}, 78(4(472)):209--210, 2023.

\bibitem[Tim11]{Timashev2011}
Dmitry~A. Timashev.
\newblock {\em Homogeneous spaces and equivariant embeddings}, volume 138 of
  {\em Encyclopaedia of Mathematical Sciences}.
\newblock Springer, Heidelberg, 2011.
\newblock Invariant Theory and Algebraic Transformation Groups, 8.

\end{thebibliography}
\end{document}